\documentclass[preprint, 12pt]{elsarticle}

\usepackage[T1]{fontenc}
\usepackage[utf8]{inputenc}
\usepackage[hmargin=0.9in]{geometry}
\usepackage[babel,german=quotes]{csquotes}
\usepackage{subcaption}
\usepackage{graphicx}
\usepackage[colorlinks=true,citecolor=blue,urlcolor=blue]{hyperref}
\usepackage{caption}
\usepackage{amsmath}          
\usepackage{amsthm}           
\usepackage{amssymb}          
\usepackage{bm, amsfonts}     
\usepackage{xcolor}
\usepackage{fixmath}
\usepackage{tikz}
\usepackage{bm}
\usepackage{makecell}

\usepackage{amsthm} 

\usepackage{stmaryrd}
\allowdisplaybreaks

\newtheorem{definition}{Definition}
\newtheorem{theorem}{Theorem}
\newtheorem{lemma}{Lemma}

\newtheorem{remark}{Remark}

 \newcommand{\D}{\mathrm{d}}

\newcommand{\pd}[2]{\frac{\partial #1}{\partial #2}}
\newcommand{\td}[2]{\frac{\D #1}{\D #2}}

\newcommand{\dx}{\mathrm{d}\mathbf{x}}

\newcommand{\dt}{\mathrm{d}t}

\newcommand{\R}{\mathbb{R}}
\newcommand{\N}{\mathbb{N}}

\newcommand{\Ni}{\mathcal N_i}

\newcommand{\Nis}{\mathcal N_i\backslash\{i\}}

\newcommand{\beq}{\begin{equation}}
\newcommand{\eeq}{\end{equation}}

\makeatletter
\def\ps@pprintTitle{%
  \let\@oddhead\@empty
  \let\@evenhead\@empty
  \def\@oddfoot{
    \footnotesize\itshape
    \ifx\@journal\@empty Elsevier
    \else\@journal\fi
    \hfill\today
  }%
  \let\@evenfoot\@oddfoot}
\makeatother

\newcommand{\est}[1]{\left\langle#1\right\rangle}

\begin{document}

  \begin{frontmatter}

    \title{{\Large  \bf Consistency and convergence of flux-corrected finite element methods  for nonlinear hyperbolic problems}}

	\author[TUD]{D. Kuzmin}
	\ead{Dmitri.Kuzmin@math.tu-dortmund.de}
	\author[JG]{M.  Luk{\'a}{\v{c}}ov{\'a}-Medvid’ov{\'a}}
	\ead{lukacova@mathematik.uni-mainz.de}
	\author[JG]{P.  \"Offner\corref{cor1}}
	\ead{mail@philippoeffner.de}
\cortext[cor1]{Corresponding author}
	\address[TUD]{Institute of Mathematics, Technical University Dortmund, Dortmund, Germany}
	\address[JG]{Institute of Mathematics, Johannes Gutenberg University, Mainz, Germany }

     \begin{abstract}
       We investigate the consistency and convergence of flux-corrected finite element approximations in the context of nonlinear hyperbolic conservation laws. In particular, we focus on a monolithic convex limiting approach and prove a Lax--Wendroff-type theorem for the corresponding semi-discrete problem. A key component of our analysis is the use of a weak estimate on bounded variation, which follows from the semi-discrete entropy stability property of the method under investigation. For the Euler equations of gas dynamics, we prove the weak convergence of the flux-corrected finite element scheme to a  dissipative weak solution.
      If a strong solution exists, the sequence of numerical approximations converges strongly to the strong solution.
             \end{abstract}

     \begin{keyword}
Finite Element Methods; Monolithic Convex Limiting; Lax-Wendroff Theorem; Convergence; Dissipative Weak Solutions; Hyperbolic Conservation Laws; Euler Equations
     \end{keyword}

  \end{frontmatter}

\section{Introduction}
\label{sec:intro}


Hyperbolic conservation laws are omnipresent in many problems arising in science and engineering. They describe time evolution of conserved quantities, such as mass, energy, momentum or concentration of chemical species. Because of their practical importance, nonlinear hyperbolic problems have been subject of extensive analytical and numerical studies; see, e.g.,  Dafermos~\cite{dafermosbook}, Warnecke~\cite{warneckebook}, Smoller~\cite{smollerbook}, Abgrall and Shu~\cite{handbook1,handbook2} and the references therein.

The range of states that exact solutions of a hyperbolic initial value problem may attain is often known to be a subset of a convex \emph{invariant domain}. For the Euler equations of gas dynamics, this generalized maximum principle guarantees positivity of the density and internal energy. Moreover, the validity of entropy inequalities and entropy minimum principles can be shown for admissible solutions. A fundamental criterion for the design of robust numerical methods is preservation of physically relevant properties at the (semi-)discrete level. Guermond and Popov \cite{guermond2016} showed that this requirement is met for a continuous finite element counterpart of the local Lax--Friedrichs method. Recent years have witnessed significant advances in the development of high-order extensions that use the framework of \emph{algebraic flux correction} \cite{afc2} to enforce invariant domain preservation \cite{guermond2018,convexCG,convexHO} and entropy stability \cite{entropyFD,entropySD,entropyHO}. A review of such high-resolution schemes and of the underlying theory can be found in \cite{bookWS}.

Convergence analysis of numerical methods for nonlinear hyperbolic problems has typically been restricted to scalar conservation laws; see, e.g., Kr\"oner~\cite{kroenerbook}, Feistauer~\cite{feistauerbook}, Coquel et al.~\cite{coquel1991, coquel1993} and the references therein. As demonstrated in a pioneering work of
Tadmor~\cite{tadmor}, a  crucial property for nonlinear stability analysis of numerical schemes for hyperbolic conservation laws is the discrete entropy inequality. Recently, a breakthrough in the convergence analysis of numerical methods for systems of hyperbolic conservation laws has been achieved by considering very weak, so-called measure-valued solutions. For details we refer to  Fjordholm, Mishra, Tadmor et al. \cite{fjordholm_3,FKMT17} and Feireisl, Luk\'a\v{c}ov\'a-Medvid'ov\'a et al. \cite{feireisl2020b, FLM20, feireisl2021}.

The aim of the present paper is to analyze consistency and convergence of flux-corrected finite element methods that use \emph{monolithic convex limiting} techniques (as proposed by Kuzmin et al.~\cite{convexCG,convexHO,entropyFD, entropySD}) to ensure nonlinear stability. Two main results are as follows:
\begin{itemize}
\item Using the entropy stability property, we derive weak BV estimates that enable us to prove a Lax--Wendroff-type theorem for spatial
  semi-discretizations; see Section~4.

\item For the multi-dimensional Euler equations of gas dynamics, we prove weak convergence of flux-corrected approximations to dissipative weak solutions; see
Section~5.
\end{itemize}

We conclude this introductory section by specifying our notation. In what follows, we consider nonlinear
hyperbolic initial value problems that can be written in the compact form
\begin{subequations}\label{ibvp}
\begin{alignat}{3}
  \pd{u}{t}+\nabla\cdot\bm{f}(u)&=0 &&\quad\mbox{in}\ \Omega\times (0,T),
    \label{ibvp-pde}\\
     u&=u_0 &&\quad\mbox{in}\ \Omega\ \mbox{at}\ t=0.
\end{alignat}
\end{subequations}
Here $u(\mathbf{x},t)$ is a vector of $m$ conserved quantities,
$\bm{f}(u)$ is an  array of $m$ flux functions, and
$u_0$ is the initial data. On the boundary $\partial\Omega$ of a domain $\Omega\subset\R^d$, $d =1,2,3$,
we prescribe periodic boundary conditions.

A convex set $\mathcal G\subset\R^m$ is called an \emph{invariant domain} if any exact solution of
\eqref{ibvp} satisfies
$$u(\mathbf{x},t)\in\mathcal G\qquad \forall (\mathbf{x},t)\in \bar\Omega\times [0,T].
$$

Let $\eta(u)$  be a convex function and $\bm{q}(u)$ be a vector field such that
$\bm q'(u)=\eta'(u)^\top\bm f'(u)$. Then $\{\eta(u),\bm{q}(u)\}$ is
called an \emph{entropy pair} and (a weak form of) the entropy inequality
\beq
\pd{\eta(u)}{t}+\nabla\cdot\bm{q}(u)\le0\quad\mbox{in}\ \Omega\times (0,T)
\eeq
holds. The components of the gradient
$v(u)=\eta'(u)$ are called \emph{entropy variables}.
For further reference, we also define the vector-valued \emph{entropy
 potential} $\boldsymbol{\psi}(u)=v(u)^\top\bm{f}(u)-\bm{q}(u)$.
\medskip

An important hyperbolic problem of the form \eqref{ibvp} is the compressible Euler system of gas dynamics. Our theoretical investigations are focused on proving convergence for this particular flow model.
 The Euler system consists of  conservation laws for the density $\rho$, momentum $\bm{m}=\varrho \bm{v}$ and total energy
$E=\frac{1}{2} \varrho |\bm{v}|^2 +\varrho e.$ Here $\bm{v}:=(v_1, \dots, v_d)^\top$ and $e$ stand for the velocity vector in $\R^d$ and  the specific internal energy, respectively. In the compact form \eqref{ibvp}, the vector of conservative variables and the flux vector functions are denoted by
$$
u=(\varrho, \bm{m}, E)^\top, \ \bm{f}_m=(\varrho v_m, v_m \bm{m} +p \bm{e}_m, v_m(E+p))^\top, m=1, \dots, d,
$$
respectively. Here $\bm{e}_m$ represents the $m$th row of the unit matrix.
To determine the pressure $p$, we use the equation of state
$p=(\gamma-1)\varrho e$ for an ideal gas with heat capacity ratio $\gamma>1$.

The mathematical entropy $\eta$ and thermodynamic entropy   $s:=\log \frac{p}{\varrho^{\gamma}}$ are related by
\begin{equation}\label{eq:entropy_convex}
\eta=- \frac{\varrho  s}{\gamma-1}.
\end{equation}
The entropy flux $\bm{q}:=(q_1, \dots, q_d)$ corresponding to $\eta$
is given by $q_m=\eta v_m$, $m=1, \dots, d$.

\section{Monolithic convex limiting}
\label{sec:limiting}

We discretize \eqref{ibvp-pde} in space using
  a continuous piecewise linear
  finite element
  approximation on a conforming mesh $\mathcal T_h$. For simplicity,
  we assume that $\bigcup_{K\in\mathcal T_h}K=\bar\Omega$.
  The space
$$V_h=\{v\in C(\bar\Omega)\,:\, v|_K\in\mathbb{P}_1(K)\ \forall
K\in\mathcal T_h\}$$ is spanned by Lagrange basis
functions $\phi_{{1}},\ldots,\phi_N$.
Introducing
the approximations
$$
u_h=\sum_{j=1}^{N_h}u_j\phi_j\approx
u,
\qquad \bm{f}_h
=\sum_{j=1}^{N_h}\bm{f}_j\phi_j
\approx \bm{f}(u_h),
$$
where $\bm{f}_j=\bm{f}(u_j)$ for $j=1,\ldots,N_h$, we
consider the semi-discrete problem
\beq
\int_{\Omega}\phi_h\left(\pd{u_h}{t}+\nabla\cdot\bm{f}_h\right)\dx=0
\qquad\forall \phi_h\in V_h.
\eeq
Using $\phi_h\in\{\phi_1,\ldots,\phi_{N_h}\}$,
  we obtain the system of differential-algebraic equations
\beq\label{galerkin-dae}
\sum_{j\in\mathcal N_i}\left(m_{ij}\td{u_j}{t}+\bm{f}_j\cdot
\bm{c}_{ij}\right)=0,
  \qquad i=1,\ldots,N_h.
 \eeq
  The integer set $\mathcal N_i$ contains the indices
  of all nodes $j$ such that the basis functions
  $\phi_i$ and $\phi_j$ have overlapping supports.
The coefficients
$m_{ij}$ and $\bm{c}_{ij}$ are defined by
$$
m_{ij}=\int_\Omega\phi_i\phi_j\dx,\qquad
\bm{c}_{ij}=\int_{\Omega}\phi_i\nabla\phi_j\dx.
$$
Using integration by parts and the assumption of periodic boundary conditions, we find that
\beq\label{prop:cijdij}
\bm{c}_{ii}=\bm{0},\qquad
\bm{c}_{ij}=-\bm{c}_{ji}\qquad
\forall i,j\in\{1,\ldots,N_h\}.
\eeq

The Lagrange basis functions possess the partition of unity
property $\sum_{j=1}^{N_h}\phi_j\equiv 1$. Therefore,
the coefficients $\bm{c}_{ij}$ of the discrete
gradient operator satisfy
\beq\label{prop:cij}
\sum_{j\in\mathcal N_i\backslash\{i\}}\bm{c}_{ij}
=\sum_{j\in\mathcal N_i}\bm{c}_{ij}=\bm{0}.
\eeq
 It follows that
an equivalent representation of system
\eqref{galerkin-dae} is given by (see, e.g., \cite[Ch.~6]{bookWS})
\beq\label{galerkin-split}
m_i\td{u_i}{t}=\sum_{j\in\mathcal N_i\backslash\{i\}}
[d_{ij}(u_j-u_i)-(\bm{f}_j-\bm{f}_i)
  \cdot\bm{c}_{ij}+f_{ij}],
\eeq
where $m_i=\sum_{j\in\Ni}m_{ij}=\int_\Omega\phi_i\dx>0$ is a diagonal entry of the
lumped mass matrix and $d_{ij}$ is an artificial viscosity
coefficient (to be defined below). The antidiffusive flux
$$
f_{ij}=m_{ij}(\dot u_i-\dot u_j)+d_{ij}(u_i-u_j)=-f_{ji},\qquad
j\in\mathcal N_i\backslash\{i\}
$$
corrects the mass lumping error and offsets the diffusive flux
$d_{ij}(u_j-u_i)$ in \eqref{galerkin-split}.
In view of~\eqref{galerkin-dae}, the
nodal time derivatives $\dot u_j:=\td{u_j}{t}$ can be
calculated by solving linear system
$$
\sum_{j\in\mathcal N_i}m_{ij}\dot u_j=-\sum_{j\in\mathcal N_i}\bm{f}_j\cdot
\bm{c}_{ij},
  \qquad i=1,\ldots,N_h.
  $$

 Let $\lambda(u_L,u_R,\bm{n})$ denote an upper bound for
 the maximum speed of a Riemann problem with the flux function
 $\bm{f}(u)\cdot\bm{n}$ and initial states
 $u_L,u_R\in\mathcal G$. In the low-order method
 analyzed by Guermond and Popov \cite{guermond2016},
 the artificial viscosity coefficients
 $$
 d_{ij}=\begin{cases}
 \max(\lambda_{ij}|\bm{c}_{ij}|,\lambda_{ji}|\bm{c}_{ji}|) & \mbox{if}\ j\in \Nis,\\
\sum_{k\in\Nis}d_{ik} & \mbox{otherwise}
\end{cases}
 $$
 are defined using
 $\lambda_{ij}=
 \lambda(u_i,u_j,\bm{c}_{ij}/|\bm{c}_{ij}|)$, where
 $|\cdot|$ denotes the Euclidean norm in $\R^d$.

The flux-corrected space
discretizations that we consider in this work replace
\eqref{galerkin-split} with
\begin{align}
m_{i}\td{u_i}{t}&=
\sum_{j\in\mathcal N_i\backslash\{i\}}[d_{ij}(u_j-u_i)-
  (\bm{f}_j-\bm{f}_i)\cdot \bm{c}_{ij}+f_{ij}^*]\nonumber\\
&= \sum_{j\in\mathcal N_i\backslash\{i\}}[2d_{ij}(\bar u_{ij}-u_i)+f_{ij}^*]=
 \sum_{j\in\mathcal N_i\backslash\{i\}}2d_{ij}(\bar u_{ij}^*-u_i),\label{mcl}
\end{align}
where $f_{ij}^*=-f_{ji}^*$ is an inequality-constrained approximation to the target flux $f_{ij}$.
We calculate $f_{ij}^*$ using the monolithic convex limiting (MCL)
algorithm proposed in \cite{convexCG}. If $\mathcal G$ is an invariant domain containing
the states $u_i$ and $u_j$, then the MCL approach guarantees that
\beq\label{property-idp}
u_i,u_j\in\mathcal G\quad\Rightarrow\quad
\bar u_{ij}=\frac{u_j+u_i}{2}-\frac{
  (\bm{f}_j-\bm{f}_i)\cdot\bm{c}_{ij}}{2d_{ij}}\in \mathcal G\quad\Rightarrow\quad
\bar u_{ij}^*=\bar u_{ij}+\frac{f_{ij}^*}{2d_{ij}}\in \mathcal G.
\eeq

Sufficient conditions for entropy stability of finite difference and finite volume schemes were formulated by Tadmor~\cite{tadmor87,tadmor} and  Ray et al.~\cite{ray}. Adapting their analysis to the finite element setting, Kuzmin and Quezada de Luna \cite{entropySD} found that the semi-discrete scheme \eqref{mcl} is entropy conservative or dissipative w.r.t. an entropy pair $\{\eta(u),\bm{q}(u)\}$ if
\beq\label{estab}
  \frac{(v_i-v_j)^\top}{2}[d_{ij}(u_j-u_i)-
   (\bm{f}_j+\bm{f}_i)\cdot\bm{c}_{ij}+f_{ij}^*]
  \le (\boldsymbol{\psi}_j-\boldsymbol{\psi}_i)\cdot\bm{c}_{ij}.
  \eeq
We limit the fluxes $f_{ij}^*$ in a way that ensures the validity of this stability condition \cite{bookWS,entropyFD,entropySD}.

\section{Stability analysis}
\label{sec:stability}

The following theorems guarantee that MCL scheme \eqref{mcl} is invariant domain preserving (IDP) and entropy stable if conditions
\eqref{property-idp} and \eqref{estab} hold for $f_{ij}^*=-f_{ji}^*$.

  \begin{theorem}[IDP criterion for spatial semi-discretizations \cite{timelim}]
 \label{thm:IDP}   Let
    $$G=\{v\in\R^N:\ v_i\in [\alpha,\beta],\ i=1,\ldots,N\}.$$
    Consider an initial value problem of the form
  \begin{align} \label{timelim:ivp}
    \td{u_i}{t}  = a_i(u)(g_i(u) - u_i),\qquad  u_i(0) & = u_i^0,
    \qquad i=1,\ldots,N.
  \end{align}
  Suppose that \eqref{timelim:ivp}
  has a unique solution for $t\in[0,T]$. Assume that $u(0)\in G$ and
  \[
  u\in G\quad\Rightarrow\quad g_i(u)\in [\alpha,\beta],\quad
   0 \le a_i(u) \le C\quad \forall i\in\{1,\ldots,N\},
   \]
where $C>0$ is independent of $u$. Then $u(t)\in G$ for all $t\in[0,T]$.
  \end{theorem}

  \begin{proof} See \cite[Theorem~1]{timelim}.
  \end{proof}

  \begin{remark}
    The assumption of well-posedness can be replaced with the requirement that the right-hand side
    of the (nonlinear) ODE system \eqref{timelim:ivp} be a Lipschitz-continuous function \cite{timelim}.
  \end{remark}

    For our semi-discrete MCL scheme \eqref{mcl}, the functions $a_i(u)$ and $g_i(u)$ are defined by
    $$
a_i(u)= \sum_{j\in\mathcal N_i\backslash\{i\}}2d_{ij},\qquad g_i(u)=\frac{1}{a_i(u)} \sum_{j\in\mathcal N_i\backslash\{i\}}2d_{ij}\bar u_{ij}^*.
$$

If the discretization in time is performed using an explicit strong stability preserving (SSP) Runge--Kutta method, then the forward Euler stages
\beq\label{uispp}
u_i^{\rm SSP}=u_i+\frac{\Delta t}{m_i}
\sum_{j\in\mathcal N_i\backslash\{i\}}2d_{ij}(\bar u_{ij}^*-u_i)
=\left(1-\frac{\Delta ta_i(u)}{m_i}\right)u_i+
\left(\frac{\Delta ta_i(u)}{m_i}\right)g_i(u)
\eeq
produce convex combinations of $u_i$ and $g_i(u)$
under the time step restriction
\beq\label{uispp-cfl}
\Delta ta_i(u)\le m_i\qquad \forall i\in\{1,\ldots,N\}.
\eeq
 This proves the conditional IDP property of the fully discrete scheme.
For the case $\bar u_{ij}^*=\bar u_{ij}$, such
a convexity-based proof was presented
by Guermond and Popov \cite{guermond2016}.
\bigskip

In the scalar ($m=1$) case, the semi-discrete scheme  \eqref{mcl} is
$L^\infty$ stable by Theorem~\ref{thm:IDP}. For a nonlinear system (such as the Euler equations, which will be discussed and analyzed in Section~\ref{sec:convergence}), the intermediate states $\bar u_{ij}$ coincide with spatially averaged exact solutions of one-dimensional Riemann problems \cite{guermond2016,harten1983}. If the initial value problem under consideration has a local invariant domain $\mathcal G_T$ that consists of bounded states, then the $L^\infty$ stability of \eqref{uispp} is guaranteed for initial data belonging to $\mathcal G_T$ and time steps satisfying \eqref{uispp-cfl}. Dividing \eqref{uispp} by $\Delta t$ and passing to the limit $\Delta t\to 0$, one can then show  the $L^\infty$ stability of \eqref{mcl}. In particular, for the Euler equations of gas dynamics we will show that the desired $L^\infty$ stability follows from a natural assumption that numerical solutions stay in a non-degenerate region, see Section~\ref{sec:convergence} for further details.

\begin{theorem}[On semi-discrete entropy stability of flux-corrected FEM]\label{thm:EntropyTadmor}
  Let $\{\eta(u),\bm{q}(u)\}$ be
an entropy pair of the hyperbolic problem \eqref{ibvp}. Suppose that condition \eqref{estab}
holds for all indices $j\in\Ni\backslash\{i\}$. Then
a solution to \eqref{mcl} satisfies the entropy inequality
\beq\label{semi-ineq}
m_i\td{\eta(u_i)}{t}\le\sum_{j\in\Nis}[g_{ij}-
   (\bm{q}_j-\bm{q}_i)\cdot\bm{c}_{ij}],
\eeq
where
\beq\label{semi-ineq-fluxes}
 g_{ij}=\frac{(v_i+v_j)^\top}2[d_{ij}(u_j-u_i)+f_{ij}^*]
-\frac{(v_i-v_j)^\top}2(\bm{f}_j-\bm{f}_i)\cdot\bm{c}_{ij}.
\eeq
\end{theorem}

\begin{proof}
See \cite{bookWS,entropySD,entropyHO}.
\end{proof}

If the backward Euler method is used to discretize \eqref{mcl}
 in time, and the assumptions of Theorem \ref{thm:EntropyTadmor}
hold, then the fully discrete scheme is
entropy stable for
any time
step. Otherwise, fully discrete
entropy stability can again be enforced using a limiter-based fix \cite{bookWS,entropyFD}.

In the finite difference context, Coquel and LeFloch \cite{coquel1991} found for scalar hyperbolic conservation laws
that entropy stability implies uniform boundedness of discrete space
derivatives. This finding enabled them to prove convergence of
flux-corrected discretizations for conservation laws in several
space dimensions \cite{coquel1991,coquel1993}. It turns out that
the semi-discrete entropy stability of our MCL scheme also prevents
unbounded growth of weak derivatives. We derive an upper bound for
the rate of entropy dissipation in the following lemma, in which
$d_{ij}^{\min}$ denotes a solution of the scalar equation (cf. \cite{entropySD})
$$
\frac12(v_i-v_j)^\top[d_{ij}^{\min}(u_j-u_i)-
  (\bm{f}_j+\bm{f}_i)\cdot \bm{c}_{ij}]=
(\boldsymbol{\psi}_j-\boldsymbol{\psi}_i)\cdot\bm{c}_{ij}
$$
and $\|\cdot\|_2:\R^m\to\R^+_0$ stands for the Euclidean vector norm in $\R^m$.

\begin{lemma}\label{lemma1}
  Choose an entropy pair such that $\eta(u)$ is strictly
  convex and nonnegative.
Use a limiter that produces $f_{ij}^*=\alpha_{ij}d_{ij}(u_i-u_j)$ with
$d_{ij}>\max\{0,d_{ij}^{\min}\}$ and
$\alpha_{ij}\in [0,1]$ such that
  $$ (1-\alpha_{ij})d_{ij}>d_{ij}^{\min}
    \qquad\forall j\in\Nis.
    $$
    Then for any finite time $T>0$, a solution of the semi-discrete problem \eqref{mcl} satisfies
    \begin{equation}
    \label{BV}
    \int_0^T\left(
\sum_{i=1}^{N_h}\sum_{j\in\Nis}|\bm{c}_{ij}|\|u_j-u_i\|_2^2\right)\dt\le C_T,
\end{equation}
where $C_T>0$ is a constant depending on~$T$.

\end{lemma}

\begin{proof}
  In view of \eqref{prop:cij},
  the flux-corrected evolution equation for
  $u_i$ can be written as
  \beq \label{mcl-lemma1}
m_{i}\td{u_i}{t}=
\sum_{j\in\mathcal N_i\backslash\{i\}}[(1-\alpha_{ij})d_{ij}(u_j-u_i)-
  (\bm{f}_j+\bm{f}_i)\cdot \bm{c}_{ij}].
\eeq

Multiplying it by the entropy variable $v_i=\frac12(v_i-v_j)+\frac12(v_i+v_j)$, we find that
\beq\label{eta:ode}
m_{i}\td{\eta(u_i)}{t}=
m_{i}v_i^\top\td{u_i}{t}=
\sum_{j\in\mathcal N_i\backslash\{i\}}[q_{ij}^-+q_{ij}^+],
\eeq
where
$$
q_{ij}^\pm=\frac12(v_i\pm v_j)^\top[(1-\alpha_{ij})d_{ij}(u_j-u_i)-
  (\bm{f}_j+\bm{f}_i)\cdot \bm{c}_{ij}].
$$
Since $\bm{c}_{ij}=-\bm{c}_{ji}$ by \eqref{prop:cijdij}, the so-defined
increments $q_{ij}^\pm$ satisfy
$$
q_{ij}^-=q_{ji}^-,\qquad q_{ij}^+=-q_{ji}^+\qquad \forall j\in\mathcal N_i\backslash\{i\}.
$$

The total entropy of $u_h(\cdot,t)$ at time $t\in[0,T]$ is given by
$$
\eta_\Omega(t)=\sum_{i=1}^{N_h}m_i\eta(u_i(t)).
$$
Summing equations \eqref{eta:ode} over $i=1,\ldots,N$
and integrating over $[0,T]$, we obtain
$$\eta_\Omega(T)=\eta_\Omega(0)+\int_0^T
\left(\sum_{i=1}^{N_h}
\sum_{j\in\mathcal N_i\backslash\{i\}}q_{ij}^-\right)\dt.
$$

Let $d_{ij}^{\rm add}=(1-\alpha_{ij})d_{ij}-d_{ij}^{\min}$.
Note that $d_{ij}^{\rm add}> 0$ by assumption. By \eqref{prop:cij}, we have
$$
\sum_{j\in\mathcal N_i\backslash\{i\}}
(\boldsymbol{\psi}_j-\boldsymbol{\psi}_i)\cdot\bm{c}_{ij}
=\sum_{j\in\mathcal N_i\backslash\{i\}}
(\boldsymbol{\psi}_j+\boldsymbol{\psi}_i)\cdot\bm{c}_{ij}=0.
$$
Using this result and the representation
$$
q_{ij}^-=\frac{d_{ij}^{\rm add}}2(v_i-v_j)^\top (u_j-u_i)+
(\boldsymbol{\psi}_j-\boldsymbol{\psi}_i)\cdot\bm{c}_{ij},
$$
we arrive at the identity
$$
\eta_\Omega(T)=\eta_\Omega(0)+\int_0^T\left(
\sum_{i=1}^{N_h}
\sum_{j\in\mathcal N_i\backslash\{i\}}\frac{d_{ij}^{\rm add}}2(v_i-v_j)^\top (u_j-u_i)
\right)\dt.
$$
The Taylor expansion of $v(u)=\eta'(u)$ about $u_i$ reveals that
$v_j-v_i=\eta''(\hat u_{ij})(u_j-u_i)$ for
a convex combination $\hat u_{ij}$ of the states $u_i$ and $u_j$.
The symmetric entropy Hessian
$\eta''(\hat u_{ij})$ is positive definite by virtue of the assumption
that $\eta(u)$ is strictly convex. It follows that
$$
(v_j-v_i)^\top (u_j-u_i)\ge \mu_{ij}\|u_j-u_i\|_2^2,
$$
where $\mu_{ij}>0$ is the minimum eigenvalue of $\eta''(\hat u_{ij})$.
Thus
$$0\le \int_0^T\left(
\sum_{i=1}^{N_h}
\sum_{j\in\mathcal N_i\backslash\{i\}}\frac{d_{ij}^{\rm add}}2
\mu_{ij}\|u_j-u_i\|_2^2
\right)\dt\le \eta_\Omega(0)-\eta_\Omega(T)\le \eta_\Omega(0).
$$
By definition of $d_{ij}$ and  $d_{ij}^{\min}$, the nonnegative entropy viscosity coefficients
$d_{ij}^{\rm add}$ are proportional to $|\bm{c}_{ij}|$.
Hence, the last estimate proves the validity of the claim.
\end{proof}

\begin{remark}
  The above proof follows the derivation of ``weak BV'' estimates
  for entropy-stable finite volume schemes in Fjordholm et al.~\cite{fjordholm, fjordholm_3} and Feireisl, Luk\'a\v{c}ov\'a-Medvid'ov\'a
  et al.~\cite{feireisl2020b}.
\end{remark}

The following result explains the practical significance of Lemma
\ref{lemma1}.

\begin{lemma}\label{lemma2}
 Under the assumption that the sequence of meshes $\{\mathcal T_h\}_{h \searrow 0}$ is shape regular, there exist constants $C_1,C_2>0$ independent of $h$ such that
 $$C_1h|v_h|_{H^1(\Omega)}^2\le d_h(v_h,v_h)
 \le C_2h|v_h|_{H^1(\Omega)}^2\qquad\forall v_h\in V_h,
 $$
 where  $d_h(\cdot,\cdot)$ is defined by
   $$
  d_h(v_h,w_h)=\sum_{i=1}^{N_h}\sum_{j\in\Nis}|\bm{c}_{ij}|
  (v_j-v_i)^\top
  (w_j-w_i)\qquad \forall v_h,w_h\in V_h.
  $$

  \end{lemma}

\begin{proof}
  Let $\mathcal N_K$ contain the indices of nodes belonging to
 an element $K\in\mathcal T_h$. Denote the
local mesh size by $h_K$.
 Construct $b_h(v_h,v_h)=
\sum_{K\in\mathcal T_h}b_K(v_h,v_h)$ using the local bilinear forms
  $$b_K(v_h,w_h)
=h_K^{d-1}\sum_{i\in\mathcal N_K}\sum_{j\in\mathcal N_K\backslash\{i\}}(v_j
  -v_i)^\top(w_j-w_i).
$$

  A straightforward generalization of the result obtained by
  Guermond and Popov \cite[Lemma 2.2]{guermond2016b} for
  scalar quantities ($m=1$) to the case of $m\ge 1$ variables shows that
$$
c_1|v_h|_{H^1(\Omega)}^2
\le \frac{h_K^{1-d}\|J_K^{-1}\|^{2}}{|\det J_K^{-1}|}
 b_K(v_h,v_h)\le c_2|v_h|_{H^1(\Omega)}^2,
    $$
    where $J_K$ is
    the Jacobian of the affine mapping from the reference element to $K$. Note that
    $|\det J_K^{-1}|=\mathcal
    O(h_K^{-d})$ and $\|J_K^{-1}\|=\mathcal O(h_K^{-1})$ under the assumption
    of shape regularity.
    The seminorm induced by $d_h(\cdot,\cdot)$ is equivalent to
    that induced by $b_h(\cdot,\cdot)$
because $|\bm{c}_{ij}|=\mathcal O(h^{d-1})$.
    Hence, the claim is true.
  \end{proof}

  \section{Semi-discrete Lax--Wendroff theorem}
\label{sec:consistency}

A finite element version of the Lax--Wendroff theorem for fully discrete schemes was proven in \cite{entropyFD}. In this section, we use similar arguments to prove semi-discrete consistency of our flux-corrected finite element method \eqref{mcl}. For simplicity, we restrict our analysis to the scalar ($m=1$) case. However, the following theorem admits a straightforward extension to systems.
\begin{theorem}[On semi-discrete Lax--Wendroff consistency of flux-corrected FEM]\label{TH:LW}
  Suppose that $u_0 \in H^2(\Omega)$ and that the flux function $\bm f(u)\in C^0(\R)^d$ is Lipschitz with constant $\lambda > 0$.

  Define $u_{h_k}(\cdot, 0) = I_{h_k} u_0$ using the interpolation operator $I_{h_k}$ and compute numerical solutions $u_{h_k}$ using the semi-discrete scheme \eqref{mcl}.
 Assume that there exist a finite time $T > 0$, a~function $u \in L^2(\Omega \times (0,T))$,  and a constant $C_{\mathcal G}>0$ independent of $k\in\N$ such that
 \begin{gather}\label{EQ:u_conv}
  \| u_{h_k} - u \|_{L^2(\Omega \times (0,T))} \to 0 \quad \text{as}\ k \to \infty,\qquad
  \| u_{h_k} \|_{L^\infty(\Omega \times (0,T))}  \le C_{\mathcal G}.
 \end{gather}
 If the assumptions of Lemmas \ref{lemma1} and \ref{lemma2}
 hold, then $u$ is a weak solution
 of \eqref{ibvp} in the sense that
 \begin{equation}\label{weakcont}
  \int_0^T \int_\Omega \left[\tfrac{\partial \phi}{\partial t} u + \nabla \phi\cdot \bm f(u)\right] \dx \dt + \int_\Omega \phi(\bm x, 0) u(\bm x, 0) \dx =0
 \end{equation}
 for all test functions
 $\phi \in C^2_c(\bar\Omega \times [0,T])$ (compact support in time, periodic in space).

Moreover, the weak entropy inequality
 \begin{equation}\label{EQ:weak_ent}
  \int_0^T \int_\Omega \left[\tfrac{\partial\phi }{\partial t} \eta(u) + \nabla \phi\cdot \bm{q}(u)\right] \dx \dt + \int_\Omega \phi(\mathbf x, 0) \eta(u(\mathbf x, 0)) \dx\ge 0
 \end{equation}
 holds for all nonnegative  test functions $\varphi \in C^2_c(\bar\Omega \times [0,T])$ and any entropy pair $\{\eta(u),\bm{q}(u)\}$ that meets the assumptions of  Lemma \ref{lemma1}.

\end{theorem}
\begin{proof}
  The interpolation operator $I_h:C(\bar\Omega)\to V_h$ approximates $v\in C(\bar\Omega)$ by
  a piecewise-linear function $I_hv=\sum_{i=1}^{N_h}v_i\phi_i\in V_h$
  such that $v_i=v(\mathbf{x}_i)$ and (see, e.g., \cite{larsson2003})
  \beq\label{interpolation_error}
  \|I_hv-v\|_{L^2(K)}\le Ch^2\|v\|_{H^2(K)},\qquad
   |I_hv-v|_{H^1(K)}\le Ch\|v\|_{H^2(K)}\qquad \forall K\in\mathcal T_h.
   \eeq
   Notice that
   $$\int_\Omega I_h(v_hu_h)\dx=
   \sum_{i=1}^{N_h}m_iu_iv_i,\quad
   \int_\Omega v_h\nabla\cdot (I_h\bm{f}(u_h))\dx=\sum_{i=1}^{N_h}v_i
   \sum_{j\in\Ni}\bm{f}_j\cdot\bm{c}_{ij}\qquad\forall u_h,v_h\in V_h.
$$
  Let $\phi \in C^2_c(\bar\Omega \times [0,T])$ and
    $\phi_h(\cdot,t)=I_h\phi(\cdot,t)$ for $t\in[0,T]$. To show consistency with
\eqref{weakcont},
    we multiply \eqref{mcl} by $\phi_i(t)=\phi_h(\mathbf{x}_i,t)$,
    sum over $i=1,\ldots,N_h$, and integrate in time over $[0,T]$. Next, we
    perform integration by parts and multiply the resulting equation by
    $-1$. This gives
    \beq\label{weakdiscr}
    \int_0^T \int_\Omega \left[\pd{\phi_h}{t} u_h + \nabla \phi_h\cdot \bm f(u_h)\right] \dx \dt + \int_\Omega \phi_h(\bm x, 0) u_h(\bm x, 0) \dx =R_h(u_h,\phi_h),
 \eeq
 where
   \begin{equation}\label{eq_consistency_error}
   \begin{aligned}
     R_h(u_h,\phi_h)
     &=\underbrace{\int_0^T\int_\Omega\left[\pd{\phi_h}{t} u_h -
         I_h\left(\pd{\phi_h}{t}u_h\right)\right]\dx\dt}_{R_h^1(u_h,\phi_h)}
     +\underbrace{\int_0^T \int_\Omega     \nabla \phi_h\cdot
     [\bm f(u_h)-I_h\bm f(u_h)] \dx \dt }_{R_h^2(u_h,\phi_h)}\\
     &+\underbrace{\int_0^T\left[\sum_{i=1}^{N_h}\phi_i(t)\sum_{j\in\Nis}
         (1-\alpha_{ij}(u_h))d_{ij}(u_i(t)-u_j(t))\right]\dt}_{R_h^3(u_h,\phi_h)}
         \end{aligned}
   \end{equation}
   is the consistency error caused by the use of inexact quadrature and algebraic
   stabilization\footnote{$R_h(u_h,\phi_h)=0$ for the standard Galerkin discretization.}.

   We need to show that
   \begin{subequations}
   \begin{align}
    \lim_{h\to 0}\int_0^T \int_\Omega\pd{\phi_h}{t} u_h\dx\dt&=\int_0^T
    \int_\Omega \tfrac{\partial \phi}{\partial t} u \dx\dt,\label{conda}\\
    \lim_{h\to 0} \int_0^T
    \int_\Omega \nabla \phi_h\cdot \bm f(u_h)\dx&=
    \int_\Omega \nabla \phi\cdot \bm f(u) \dx\dt,\label{condb}\\
    \lim_{h\to 0} \int_0^T
    \int_\Omega \phi_h(\bm x, 0) I_hu_0(\bm x) \dx
    &= \int_\Omega \phi(\bm x, 0) u_0(\bm x)\dx\dt,\label{condc}\\
    \lim_{h\to 0}R_h(u_h,\phi_h)&=0.\label{condd}
   \end{align}
   \end{subequations}
   Let us begin with  \eqref{condc}. We have
   $$
   \int_\Omega \phi_h(\bm x, 0) I_hu_0(\bm x) \dx=
    \int_\Omega \phi_h(\bm x, 0)u_0 \dx+
   \int_\Omega \phi_h(\bm x, 0)(I_hu_0(\bm x)-u_0(\bm x)) \dx,
   $$
   where $u_0\in H^2(\Omega)$ and, therefore, $u_0\in C(\bar\Omega)$ for $d=\{1,2,3\}$. The first error estimate in \eqref{interpolation_error} implies that
   $\|I_hu_0-u_0\|_{L^2(\Omega)}=\mathcal O(h^2)$ and
   $\|I_h\phi(\cdot,0)-\phi(\cdot,0)\|_{L^2(\Omega)}=\mathcal O(h^2)$,
   where $I_h\phi(\cdot,0)=\phi_h(\cdot,0)$ by definition of $\phi_h$.
   This proves the validity of \eqref{condc}. From
   $$
   \int_\Omega\pd{\phi_h}{t} u_h\dx=\int_\Omega\pd{\phi_h}{t} u\dx
   +\int_\Omega\pd{\phi_h}{t} (u_h-u)\dx
   $$
   and the assumption that   $\| u_{h} - u \|_{L^2(\Omega \times (0,T))} \to 0$ as $h\to 0$,
   we deduce the validity of  \eqref{conda}. Using the assumed Lipschitz continuity
   property $|\bm f(u_h)-\bm f(u)|\le \lambda|u_h-u|$ and the estimate
$$|\phi_h(\cdot,t)|_{H^1(\Omega)}\le
   |\phi(\cdot,t)|_{H^1(\Omega)}+|I_h\phi(\cdot,t)- \phi(\cdot,t)|_{H^1(\Omega)}
=   |\phi(\cdot,t)|_{H^1(\Omega)}+\mathcal O(h),$$
we find that the second integral on the right-hand side of the identity
\begin{equation}\label{eq_flux_estimate}
   \int_\Omega \nabla \phi_h\cdot \bm f(u_h)\dx
   =\int_\Omega \nabla \phi_h\cdot \bm f(u)\dx
   +\int_\Omega \nabla \phi_h\cdot(\bm f(u_h)-\bm f(u))\dx
   \end{equation}
 vanishes, while the first one converges to the right-hand side of \eqref{condb} as $h\to 0$.

 It remains to estimate the three components of the consistency error
 $R_h(u_h,\phi_h)$. The term

 $$R_h^1(u_h,\phi_h)=\int_0^T\int_\Omega\left[\pd{\phi_h}{t} u_h -
   I_h\left(\pd{\phi_h}{t}u_h\right)\right]\dx\dt
 $$
 is the error due to mass lumping. Let $\dot \phi_i=\pd{\phi_i}{t}$
 and $m_{ij}=\int_\Omega\phi_i\phi_j\dx$. Define the bilinear form
 $$
m_h(v_h,w_h)=\sum_{i=1}^{N_h}\sum_{j\in\Nis}m_{ij}(v_j-v_i)(w_j-w_i)\qquad\forall v_h,w_h\in V_h.
 $$
 Replacing $|\bm{c}_{ij}|=\mathcal O(h^{d-1})$ with
 $m_{ij}=\mathcal O(h^{d})$ in Lemma \ref{lemma2}, we find that the seminorm induced by
 $m_h(\cdot,\cdot)$ is equivalent to
 $h|\cdot|_{H^1(\Omega)}$. The use of the Cauchy--Schwarz inequality
 yields
$$|R_h^1(u_h,\phi_h)|=\int_0^T|m_h(u_h,\dot\phi_h)|\dt\le\int_0^T
 \sqrt{m_h(u_h,u_h)}\sqrt{m_h(\dot\phi_h,\dot\phi_h)}\dt.
 $$
Introducing the nonlinear form
  $$
d_h^*(u_h;v_h,w_h)=\sum_{i=1}^{N_h}\sum_{j\in\Nis}(1-\alpha_{ij}(u_h))d_{ij}(v_j-v_i)(w_j-w_i)
\qquad \forall u_h,v_h,w_h\in V_h,
$$
we obtain a similar estimate for
 $$|R_h^3(u_h,\phi_h)|=\int_0^T|d_h^*(u_h;u_h,\phi_h)|\dt\le\int_0^T
 \sqrt{d_h^*(u_h;u_h,u_h)}\sqrt{d_h^*(u_h;\phi_h,\phi_h)}\dt.
 $$
 The seminorm induced by $d_h^*(\cdot,\cdot)$ is equivalent to $\sqrt{h}|\cdot|_{H^1(\Omega)}$.
 The $H^1$ seminorms of $\dot\phi_h(\cdot,t)$ and  $\phi_h(\cdot,t)$ are uniformly bounded
 on $[0,T]$.
 By
 Lemmas \ref{lemma1} and \ref{lemma2}, there exists $\tilde C_T>0$ such that
 $$
 \int_0^T\sqrt{h}|u_h(\cdot,t)|_{H^1(\Omega)}\dt\le \tilde C_T,\qquad
|R_h^1(u_h,\phi_h)|=\mathcal O(h^{3/2}),\quad
 |R_h^3(u_h,\phi_h)|=\mathcal O(h^{1/2}).
 $$

 The error due to the \emph{group finite element} approximation $\bm{f}(u_h)\approx
 I_h\bm{f}(u_h)$ is represented by
$$
 R_h^2(u_h,\phi_h)=
 \int_0^T\int_\Omega\nabla \phi_h\cdot
     [\bm f(u_h)-I_h\bm f(u_h)] \dx \dt
= \int_0^T\sum_{K\in\mathcal T_h} \nabla \phi_h|_K\cdot\int_K
     [\bm f(u_h)-I_h\bm f(u_h)] \dx\dt.
     $$
     We use the Cauchy--Schwarz inequality to show that
     $$
\left|\int_\Omega\nabla \phi_h\cdot
          [\bm f(u_h)-I_h\bm f(u_h)] \dx\right|
          \le |\phi_h|_{H^1(\Omega)}\|\bm f(u_h)-I_h\bm f(u_h)\|_{L^2(\Omega)}.
          $$
          Using  the fact that  $\sum_{j\in\mathcal N_K}\phi_j\equiv 1$ and
the Lipschitz continuity of $\bm f(u_{h})$, we obtain the estimate
  \begin{align*}
    \|\bm f(u_h)-I_h\bm f(u_h)\|_{L^2(\Omega)}^2&=
    \sum_{K \in \mathcal T_{h}} \int_K |\bm f(u_{h}) - I_h\bm f(u_{h})|^2\dx \\&=
    \sum_{K \in \mathcal T_{h}} \int_K \Big|
    \bm f(u_{h})\sum_{j\in\mathcal N_K}  \phi_j
    - \sum_{j\in\mathcal N_K} \bm f(u_j) \phi_j\Big|^2\dx
    \\ &\le\sum_{K \in \mathcal T_{h}}  \sum_{j\in\mathcal N_K} \int_K
    |\bm f(u_{h}) -  \bm f(u_j)|^2
 \phi_j^2
    \dx
    \\ &\le\lambda \sum_{K \in \mathcal T_{h}}  \sum_{j\in\mathcal N_K} \int_K
    |u_{h}-u_j|^2 \dx\le\lambda \sum_{K \in \mathcal T_{h}}
    \sum_{j\in\mathcal N_K}|K|\max_{i\in\mathcal N_K}
    |u_i-u_j|^2\\
    &\le Ch\sum_{K \in \mathcal T_{h}}b_K(u_h,u_h)=Ch\,b_h(u_h,u_h).
  \end{align*}
  The bilinear forms $b_K(\cdot,\cdot)$ and $b_h(\cdot,\cdot)$ are defined as in the
  proof of Lemma \ref{lemma2}. Recall that the seminorm induced by  $b_h(\cdot,\cdot)$
  is equivalent to $\sqrt{h}|\cdot|_{H^1(\Omega)}$. Thus
  $$
  \|\bm f(u_h)-I_h\bm f(u_h)\|_{L^2(\Omega)}\le Ch|u_h|_{H^1(\Omega)}.$$
  Invoking the weak derivative estimate, we conclude that
  $|R_h^2(u_h,\phi_h)|=\mathcal O(h^{1/2})$.
  The consistency of \eqref{weakdiscr} with \eqref{weakcont}
  follows from the convergence results for
  individual terms. Consistency with the weak
  entropy inequality \eqref{EQ:weak_ent} can be shown
  in the same way (cf.~\cite{harten1983,entropyFD}).
\end{proof}

\begin{remark}
  The assumption that
  $u_0 \in H^2(\Omega)$ could be waived by using the $L^2$ projection
  operator $P_h:L^2(\Omega)\to V_h$ instead of the interpolation operator
  $I_h:C(\bar\Omega)\to V_h$
  to define the initial data $u_h(\cdot,0)$.
\end{remark}

\begin{remark}
  The estimate $|R_h^3(u_h,\phi_h)|=\mathcal O(h^{1/2})$ is valid for
  any choice of the correction factors $\alpha_{ij}\in[0,1]$ and
  corresponds to the worst-case scenario. If $\alpha_{ij}=1+\mathcal O(h)$
  for all pairs of nodes, then $|R_h^3(u_h,\phi_h)|=\mathcal O(h^{3/2})$. Hence,
  a good flux limiter does not inhibit optimal convergence.
\end{remark}

\section{Convergence to dissipative weak solutions of the Euler equations}
\label{sec:convergence}

As demonstrated in the previous section, a Lax--Wendroff-type theorem holds for the flux-corrected finite element method \eqref{mcl} using $f_{ij}^*=\alpha_{ij}d_{ij}(u_i-u_j)$ if the choice of the correction factors $\alpha_{ij}\in[0,1]$ guarantees the validity of conditions \eqref{property-idp} and \eqref{estab}. The MCL schemes under investigation are designed to provide the desired properties.
In Theorem~\ref{TH:LW}, we assumed that the numerical solutions $u_h$ converge to $u \in L^2(\Omega \times (0,T))$ strongly as the mesh parameter $h$ tends to zero. Consequently, the limit $u$ is a weak (entropy) solution of \eqref{ibvp} in Theorem~\ref{TH:LW}. In the case of the multi-dimensional Euler equations, it may happen, however, that approximate solutions do not converge strongly. This situation typically occurs in applications to the Kelvin--Helmholtz or Richtmyer--Meshkov problems; see, e.g., \cite{ FLSY21, feireisl2021,fjordholm_3, FKMT17, lukacova2023}. Consequently, one may ask whether numerical solutions converge in a weaker sense.  Indeed, Abgrall et al.~\cite{abgrall2023}, Feireisl, Luk\'a\v{c}ov\'a-Medvid'ov\'a et al.~\cite{feireisl2021, FLM20, FLSY21}, Luk\'a\v{c}ov\'a-Medvid'ov\'a and \"Offner~\cite{lukacova2023} have recently proved that consistent and stable approximations of multi-dimensional Euler equations converge weakly to very weak, so-called dissipative weak solutions. We also refer the reader to the related work of  Fjordholm et al.~\cite{fjordholm_3, FKMT17} and further literature on measure-valued solutions.

We note that the concept of weak entropy solutions may be not appropriate if such solutions are potentially nonunique. Indeed, it is well known that the multi-dimensional Euler equations may possess infinitely many weak entropy solutions for particular choices of initial data; see~De Lellis and Sz\'ekelyhidi \cite{dlsz1, dlsz2}. When it comes to proving convergence of numerical methods, it is worthwhile to work with alternative solution concepts. The ongoing quest for such concepts was initiated in 1985 by DiPerna~\cite{DiPerna1985}, who argued that the framework of probabilistic (measure-valued) solutions may be more suitable for hyperbolic conservation laws.


Restricting ourselves to the Euler system
in this section, we use the theoretical framework developed in \cite{abgrall2023, feireisl2021, lukacova2023} to study convergence of the flux-corrected finite element method \eqref{mcl} to dissipative weak (DW) solutions.
These generalized solutions, which we formally define below, can be viewed as a natural extension of a set of consistent and stable approximations in a weak topology. This means that DW solutions satisfy the Euler equations modulo defect measures that account for potential concentrations and oscillations. A DW solution can also be interpreted as an expected value with respect to the underlying Young measure. In essence, such solutions represent the observable (and computable) scales of a given problem.

To define the DW solutions as in \cite{feireisl2021}, we need to introduce the following notation. Let
$\mathcal{M}^+(\overline{\Omega})$ denote the set of all positive  \textit{Radon measures} that can be identified with the space of all linear forms on $C_c(\overline{\Omega}).$ If $\overline{\Omega}$ is compact, then
$[C_c(\overline{\Omega})]^*=\mathcal{M}(\overline{\Omega})$. Furthermore, we denote by  $\mathcal{M}^+(\overline{\Omega}; \R^{d\times d}_{sym})$ the set of symmetric positive-definite matrix-valued measures, i.e.,
\begin{align*}
\mathcal{M}^+(\overline{\Omega}, \R^{d\times d}_{sym}) = \bigg\{& \nu \in \mathcal{M}^+(\overline{\Omega}, \R^{d\times d}_{sym})
\big|
\int_{\overline{\Omega}} \phi(\xi \otimes\xi): \operatorname{d} \nu\geq0  \text{ for any } \xi \in \R^d, \phi \in C_c(\overline{\Omega}), \phi\geq 0
 \bigg\}.
\end{align*}
We are now ready to give the following formal definition of  DW solutions; cf.~\cite{feireisl2021}:
\begin{definition}[Dissipative weak solution of the Euler equations]\label{def_dmv}
Let $\Omega\subset \R^d$ be a bounded domain.  Given an initial condition $(\varrho_0, \bm{m}_0, \eta_0)$ with $\varrho_0>0$ and $\int_{\Omega}\left[ \frac{1}{2}
\frac{|\bm{m}_0|^2}{\varrho_0} + e(\varrho_0, \eta_0)\right]\dx<\infty$, we call $(\varrho, \bm{m}, \eta)$ a \textbf{dissipative weak solution}  of the complete Euler system
with periodic or no-flux boundary conditions  if the following assumptions are met:
\begin{itemize}
\item  $\varrho  \in C_{weak}([0,T];L^{\gamma}(\Omega) )$,
$\bm{m}  \in C_{weak}([0,T];L^{\frac{2\gamma}{\gamma+1}}(\Omega;\R^d) )$, \\
$\eta \in  L^{\infty}(0, T; L^{\gamma}(\Omega)) \cap BV_{weak}([0,T];L^{\gamma}(\Omega)).$
\item  There exists a measure $\mathfrak{E} \in  L^{\infty}(0,T;\mathcal{M}^+(\overline{\Omega}))$ (energy defect)  such that the integral \textbf{energy inequality}
\begin{equation*}
 \int_{\Omega} \left[   \frac{1}{2} \frac{|\bm{m}|^2}{\varrho} + \varrho e(\varrho,\eta )\right]  ( \tau,\cdot) \dx + \int_{\Omega}   \operatorname{d}{\mathfrak{E}}(\tau)
 \leq  \int_{\Omega} \left[ \frac{1}{2} \frac{|\bm{m}_0|^2}{\varrho_0} +\varrho_0 e(\varrho_0,\eta_0) \right] \dx
\end{equation*}
is satisfied for a.a.  $0\leq \tau\leq T$.
\item The weak formulation
\begin{equation*}
\left[  \int_{\Omega} \varrho \phi \dx \right]_{t=0}^{t=\tau}= \int_{0}^\tau \int_{\Omega}
\left[ \varrho \partial_t \phi + \bm{m} \cdot \nabla_{\mathbf{x}} \phi  \right] \dx\dt
\end{equation*}
of the  \textbf{continuity equation}
holds for any $0 \leq \tau \leq T$  and any $\phi \in C^\infty(\overline{\Omega}\times [0,T])$. 
\item The integral identity
 \begin{align*}
 \left[  \int_{\Omega} \bm{m} \cdot\bm \phi\operatorname{d} \mathbf{x}\right]_{t=0}^{t=\tau}
 = &\int_{0}^\tau \int_{\Omega}
 \bigg[ \bm{m} \cdot \partial_t \bm \phi+
 1_{\varrho>0} \frac{ \bm{m} \otimes  \bm{m} }{\varrho}   : \nabla_{\mathbf{x}} \bm \phi\\+&1_{\varrho>0} p(\varrho, \eta)  \operatorname{div}_{\mathbf{x}}\bm \phi\bigg]
 \dx\dt
   + \int_{0}^\tau \int_{\Omega} \nabla_{\mathbf{x}} \bm\phi: \operatorname{d}{\mathfrak{R}},
 \end{align*}
where $
 \mathfrak{R}\in L^{\infty} \left(0,T; \mathcal{M}\left(\overline{\Omega}, \R^{d\times d}_{sym}\right) \right)$ is the Reynolds stress tensor of the
\textbf{momentum equation},
holds for any $0 \leq \tau\leq T$ and any test function $\boldsymbol  \phi\in C^\infty(\overline{\Omega}\times [0,T];\R^d)$ that
additionally satisfies $\bm \phi \cdot \bm{n}|_{\partial \Omega}=0$
in case of no-flux boundary conditions.

 \item The  weak formulation
\begin{equation*}
\begin{aligned}
\left[\int_{\Omega} \eta \phi  \operatorname{d} \mathbf{x} \right]_{t=\tau_1-}^{t=\tau_2+}
\leq&  \int_{\tau_1}^{\tau_2} \int_{\Omega} \left[
\eta \partial_t \phi +\est{\nu; 1_{\tilde{\varrho}>0} \left( \tilde{\eta} \frac{\tilde{\bm{m}}}{\tilde{\varrho}} \right) } \cdot \nabla_\mathbf{x} \phi  \right] \dx\dt\\
\eta(0^-, \cdot) =& \eta_0
\end{aligned}
\end{equation*}
 of the  \textbf{entropy  inequality}
holds for any $0\leq \tau_1 \leq \tau_2 <T$,
any $\phi \in C_c^\infty(\Omega\times (0,T)), \phi\geq 0$,
and a parametrized (Young) measure
$\{ \nu_{\mathbf{x},t}\}_{(\mathbf{x},t) \in \Omega\times (0,T)}$ such that
\begin{equation}\label{eq:Young}
\begin{aligned}
\nu \in L^{\infty}(\Omega\times (0,T);\mathcal{P}(\mathcal{F})),\ \mathcal{F}=\left\{ \tilde{\varrho} \in \R, \tilde{\bm{m}}\in \R^d, \tilde{\eta}\in \R  \right\};\\
\est{\nu, \tilde{\varrho}} =\varrho, \, \est{\nu, \tilde{\bm{m}}} =\bm{m},\, \est{\nu, \tilde{\eta}} =\eta,\\
\nu_{\mathbf{x},t}\left\{ \tilde{\varrho} \geq 0, (1-\gamma)\tilde{\eta} \geq \underline{s} \tilde{\varrho} \right\}=1 \text{ for a.a. } (\mathbf{x},t) \in \Omega\times (0,T).
\end{aligned}
\end{equation}
\item There exist constants $c_1\ge 0$ and  $c_2\ge c_1$  such that the  \textbf{defect compatibility condition}
$$
c_1 \mathfrak{E} \leq \operatorname{tr} [\mathfrak{R}] \leq c_2 \mathfrak{E}
$$
holds.
\end{itemize}
\end{definition}
\begin{remark}
A key property of a DW solution is its compatibility with a classical solution, see \cite[Theorem~5.7]{feireisl2021}. More precisely, if a DW solution $(\varrho, \bm{m}, \eta)$ satisfies
\begin{equation}\label{eq_regularity}
\varrho \in C^1(\overline{\Omega}\times [0,T]), \; \inf_{\Omega\times (0,T)} \varrho>0, \; \bm{u} \in C^1(\overline{\Omega}\times [0,T]; \R^d), \; \eta \in C^1(\overline{\Omega}\times [0,T]),
\end{equation}
 then $(\varrho, \bm{m}, \eta)$  is a classical solution of the Euler system.

In addition, the DW-strong uniqueness principle implies that if a classical solution to the Euler system exists, then
all DW solutions corresponding to the same initial data coincide with the classical solution; see \cite[Theorem~6.2]{feireisl2021}.
\end{remark}
It is worth noting that we interchanged the roles of the total energy and entropy in Definition~\ref{def_dmv}. In view of the strict positivity of density and pressure, a one-to-one mapping exists between the conservative variables $(\varrho, \bm{m}, E)$ and the entropy-conservative variables
$(\varrho, \bm{m}, \eta)$. Although the flux-corrected finite element method \eqref{mcl} evolves the conservative variables, the MCL algorithm may be configured to impose inequality constraints on $(\varrho, \bm{m}, \eta).$  

The classical {\em Lax equivalence theorem} is not applicable to nonlinear problems, but a typical proof of convergence to DW solutions relies on the assumption that the method under investigation is {\em consistent} and {\em stable} in a suitably defined sense \cite{feireisl2021}. Thus we need to show that this is the case for our flux-corrected finite element discretization~\eqref{mcl} of the Euler system.
\medskip

Let us first give an appropriate definition of {\em stability} for a sequence $(\varrho_h, \bm{m}_h, \eta_h)_{h \searrow 0}$ of numerical solutions to the Euler system.
  In what follows, we will assume that the approximations remain in a non-degenerate region, which is a physically reasonable hypothesis. More precisely, we assume that
there exist two positive constants $\underline{\varrho}$ and $\overline{E}$ such that
\begin{equation}\label{vacuum}
0 <\underline{\varrho} \leq \varrho_h(t),  \qquad E_h(t) \leq \overline{E} \ \text{ uniformly for }\  h\to 0.
\end{equation}
The imposition of an upper bound on the energy in  \eqref{vacuum} implies  that the velocity $|\bm{v}_h|$ is bounded since
$
 |\bm{v}_h|^2\leq \frac{2E_h}{\varrho_h} \leq \frac{2\overline{E}}{\underline{\varrho}}<C.
$
As explained in  \cite{feireisl2020b, lukacova2023_b}, one can also use  \eqref{vacuum} to show  that the pressure and temperature are bounded from above and below.
Consequently, \eqref{vacuum} implies that the sequence of numerical solutions is uniformly bounded in $L^\infty(\Omega \times (0,T)).$

\medskip
%

We proceed with showing {\em consistency} of the flux-corrected finite element  method \eqref{mcl}.

\begin{theorem}[Consistency of the MCL scheme for the Euler system]\label{th_consistency}
  Let $u_h= (\varrho_h, \bm{m}_h, E_h)^\top$ be a numerical solution obtained
  with \eqref{mcl}
  on a time interval $[0,T]$ using a mesh with spacing $h$ and
  the initial data $u_{h,0}$. Denote
  the consistency error w.r.t. $v\in\{\varrho, \bm{m}, \eta\}$ by~$e_{v_h}$.
  Assume that the sequence of meshes $\{\mathcal{T}_h\}_{h \searrow 0}$ is shape regular and the assumptions of Lemma \ref{lemma1} are satisfied. If, additionally, condition \eqref{vacuum} holds, then
the following assertions are true for all $\tau \in (0,T]$:
 \begin{itemize}
  \item for all $\phi \in C^{p+1}( \overline{\Omega} \times [0,T])$:
  \begin{equation}\label{eq:consistency_rho}
 \left[  \int_{\Omega} \varrho_h \phi\ \dx \right]_{t=0}^{t=\tau} =\int_0^\tau \int_{\Omega}[ \varrho_h \partial_t \phi +\bm{m}_h \cdot \nabla_{\mathbf{x}} \phi] \dx\dt
 +\int_0^\tau e_{\varrho_h} (t,\phi) \dt;
  \end{equation}
  \item for all $\bm\phi \in C^{p+1}( \overline{\Omega} \times [0,T];\R^d)$:
  \begin{equation}\label{eq:consistency_m}
  \begin{aligned}
 \left[  \int_{\Omega} \bm{m}_h\cdot\bm \phi\ \dx \right]_{t=0}^{t=\tau} =&\int_0^\tau \int_{\Omega}\left[ \bm{m}_h\cdot \partial_t \bm \phi+\frac{\bm{m}_h\otimes\bm{m}_h}{\varrho_h} :  \nabla_{\mathbf{x}}\bm\phi +p_h \operatorname{div}_{\mathbf{x}}\bm \phi\right] \dx \dt \\
 &+ \int_0^\tau e_{\bm{m}_h} (t,\bm\phi) \dt;
 \end{aligned}
  \end{equation}
  \item for all $\phi \in C^{p+1}( \overline{\Omega} \times [0,T]), \; \phi\geq 0$:
  \begin{equation}\label{eq:consistency_eta}
 \left[  \int_{\Omega} \eta_h \phi \ \dx\right]_{t=0}^{t=\tau} \leq \int_0^\tau \int_{\Omega} [\eta_h \partial_t \phi + (\eta_h \bm{v}_h) \cdot \nabla_\mathbf{x} \phi] \dx\dt + \int_0^\tau e_{\eta_h} (t,\phi) \dt;
  \end{equation}
  \item total energy is conserved, i.e.,
  \begin{equation}
  \label{eq:consistency_energy}
  \int_{\Omega} E_h(\tau)\dx =\int_{\Omega} E_{h}(0) \dx;
  \end{equation}
\item for $v\in\{\varrho, \bm{m}, \eta\}$, the
  consistency error $e_{v_h}$  tends to zero under mesh refinement
  \begin{equation}\label{eq:consistency_error}
   \|{e_{v_h}}\|_{L^1(0,T)} \to 0\ \ \text{ as } \ h\to 0.
  \end{equation}
 \end{itemize}
\end{theorem}

\begin{proof}

First, we realize that for all  $\phi \in C^{p+1}( \overline{\Omega}\times [0,T], \R^{d+2})$ we have the identity
\begin{equation}\label{eq_consistent_2}
  \left[\int_{\Omega} u_h  \ \dx \right]_{t=0}^{t=\tau}
  = \int_0^\tau \int_{\Omega} \td{}{t} \left( u_h \phi\right)\dx\dt
  = \int_0^\tau \int_{\Omega}[ u_h \partial_t \phi +  \phi \partial_t u_h]   \dx\dt.
\end{equation}
Next, we use the interpolation error estimate \eqref{interpolation_error} for the difference between the test function $\phi \in C^{p+1}( \overline{\Omega}\times [0,T], \R^{d+2})$ and its interpolant $\phi_h=I_h\phi$. This gives (cf. \cite{brenner2008})
\begin{equation}
\label{eq30}
 \int_0^\tau \int_{\Omega}   \phi \partial_t u_h   \dx \dt=
   \int_0^\tau \int_{\Omega}  \underbrace{\left(  \phi-\ \phi_h \right)}_{\mathcal{O}(h^{p+1})}  \partial_t u_h   \dx \dt
   + \int_0^\tau  \int_{\Omega}    \phi_h   \partial_t u_h  \dx\dt.
\end{equation}
The first term on the right-hand side vanishes as $h \to 0$ due to the boundedness of $\partial_t u_h .$  Indeed, the nodal states $u_i(t)=u_h(\mathbf x_i,t)$ are obtained by solving the nonlinear ODE system \eqref{mcl}. Since they are uniformly bounded by \eqref{vacuum}, and the flux function is Lipschitz continuous,  $u_h$ is a $C^1$ function in time
that is uniformly bounded in the limit $h \searrow 0.$

To estimate the second term on the right-hand side of \eqref{eq30}, we invoke
the definition  \eqref{mcl} of $\dot u_i(t)=\partial_t u_h(\mathbf x_i)$ for the MCL scheme.
Following the derivation of \eqref{weakdiscr}, we perform integration by
parts for the volume integral and find that
    \begin{equation*}
    \int_0^\tau \int_\Omega \phi_h \partial_t u_h =   \int_0^\tau \int_\Omega  \nabla \phi_h\cdot \bm f(u_h) \dx \dt  - \tilde{R}_h(u_h,\phi_h),
 \end{equation*}
where
    \begin{align*}
     \tilde{R}_h(u_h,\phi_h)=
     &\underbrace{\int_0^T \int_\Omega     \nabla \phi_h\cdot
     [\bm f(u_h)-I_h\bm f(u_h)] \dx \dt }_{\tilde{R}_h^2(u_h,\phi_h)}\\
     &+\underbrace{\int_0^T\left[\sum_{i=1}^{N_h}\phi_i(t)\sum_{j\in\Nis}
         (1-\alpha_{ij}(u_h))d_{ij}(u_i(t)-u_j(t))\right]\dt}_{\tilde{R}_h^3(u_h,\phi_h)}
   \end{align*}
    is the consistency error due to interpolation and limiting. The integrals
    denoted by $\tilde{R}_h^2(u_h,\phi_h)$ and $\tilde{R}_h^3(u_h,\phi_h)$ also
    appear in \eqref{eq_consistency_error}.  Following the
    proof of Theorem \ref{TH:LW}, we deduce
  \begin{equation}
  \label{eq:ce}
  | \tilde{R}_2(u_h, \phi_h)| =\mathcal{O}(h^{1/2}), \ \qquad \ | \tilde{R}_3(u_h, \phi_h)| =\mathcal{O}(h^{1/2}).
  \end{equation}
In particular, we have proven that the consistency errors  in the density and momentum equation tend to zero under mesh refinement.
The total energy is conserved if
\begin{equation}\label{energy_balance}
E_{\Omega(t)} \equiv \sum_{i=1}^{N_h} m_i E_i(t) = E_{\Omega(0)} \equiv \sum_{i=1}^{N_h} m_i E_i(0).
\end{equation}
The validity of \eqref{energy_balance} is obvious if the nodal values
$E_i(t)$ are evolved using
the energy equation of the semi-discrete MCL scheme
\eqref{mcl-lemma1}. Indeed, the
numerical fluxes $g_{ij}^*=(1-\alpha_{ij})d_{ij}(u_j-u_i)-
(\bm{f}_j+\bm{f}_i)\cdot \bm{c}_{ij}=-g_{ji}^*$ cancel out upon summation over
$i=1,\ldots,N_h$. In general, a flux-corrected finite element discretization of
the form \eqref{mcl} is globally
conservative because the right-hand side admits a decomposition into
$g_{ij}^*=d_{ij}(u_j-u_i)+f_{ij}^*-
(\bm{f}_j+\bm{f}_i)\cdot \bm{c}_{ij}=-g_{ji}^*$.

It remains to show consistency formulation for the entropy inequality. Arguing as in the proofs of 
\eqref{eq:consistency_rho}, \eqref{eq:consistency_m},  we can show that the consistency error in the entropy inequality vanishes as
$h \to 0.$ This concludes the proof.
\end{proof}

As shown in Section \ref{sec:stability}, our  monolithic convex limiting strategy makes the spatial semi-discretization \eqref{mcl} invariant domain preserving. In particular, the density and pressure are guaranteed to stay nonnegative. Moreover, the semi-discrete entropy inequality \eqref{semi-ineq} holds if the flux limiter enforces \eqref{estab} in addition to \eqref {property-idp}. Weak BV estimates \eqref{BV} and the applicability of Theorem \ref{th_consistency}
follow from the entropy stability of \eqref{mcl}.
%
%
Having shown the consistency and stability of our MCL scheme, we can proceed to the analysis of convergence.

\begin{theorem}[Weak convergence of the MCL scheme for the Euler system]\label{eq:theorem_weak} Consider a family $\{u_h\}_{h \searrow 0} \equiv \{(\varrho_h, \bm{m}_h, \eta_h)^\top \}_{h \searrow 0}$ of numerical solutions generated using \eqref{mcl} on a shape regular sequence of meshes $\{\mathcal{T}_h\}_{h \searrow 0}$. Assume that the assumptions of Lemma~\ref{lemma1} and condition \eqref{vacuum} hold.
Then there exists a subsequence of $u_h$ (denoted again by $u_h$) such that
\begin{equation}\label{eq_bounds}
\begin{aligned}
\varrho_h &\to \varrho \text{ weakly-(*) in } L^{\infty}(\Omega\times (0,T)),\\
\bm{m}_h &\to \bm{m} \text{ weakly-(*) in } L^{\infty}(\Omega\times (0,T); \R^d)),\\
\eta_h &\to \eta \text{ weakly-(*) in } L^{\infty}(\Omega\times (0,T)) \qquad \mbox{ as } h \to 0,
\end{aligned}
\end{equation}
where  $(\varrho, \bm{m}, \eta)$ is a DW solution of the Euler system.
Moreover, we have
\begin{eqnarray*}
&&E_h \equiv E(\varrho_h, \bm{m}_h, \eta_h) = \left[ \frac 1 2  \frac{|\bm{m}_h |^2}{\varrho_h} + \varrho_h e(\varrho_h, \eta_h) \right] \to
\left< \nu; E(\tilde{\varrho}, \tilde{\bm{m}}, \tilde{\eta}) \right> \\
&&\phantom{mmmmmmmm}\text{ weakly-(*) in }  L^\infty(0,T; \mathcal{M}(\overline{\Omega})) \qquad \mbox{ as } h \to 0.
\end{eqnarray*}
\end{theorem}
\begin{proof}[Sketch of the proof]
  The proof follows the analysis of finite volume and discontinuous Galerkin methods in  \cite[Theorem~5.1]{feireisl2021} and \cite{lukacova2023}.
The main steps can be outlined as follows:
\begin{enumerate}
\item Stability of the MCL scheme  \eqref{mcl} implies that
$ u_h = (\varrho_h, \bm{u}_h, \eta_h)^\top $ is uniformly bounded in $L^\infty(\Omega \times (0,T))$ w.r.t. $h \to 0.$
Consequently, there exists a subsequence of $u_h$ that is weakly-(*) convergent in $L^\infty(\Omega \times (0,T)).$
\item By the fundamental theorem of Young measures, the sequence $\{u_h\}_{h \searrow 0}$ generates a space-time parametrized probability measure. This Young measure, $\nu_{\mathbf{x},t}$, is a suitable tool for identifying weak limits of  (smooth) nonlinear functions of $u_h.$  This allows us to pass to the weak-(*) limit in all components of the consistency error as $h \to 0$.
\item  We identify the limit as an expected value of $(\varrho, \bm{m}, \eta)$ w.r.t. the Young measure $\nu.$ The corresponding functions of space and time satisfy the DW formulation. The energy balance is relaxed. Only the global energy inequality holds, and the energy defect arises in the process of passing to a weak limit.
\end{enumerate}
\end{proof}
Weak convergence to a DW solution is difficult to verify in numerical experiments. Therefore, it is appropriate to approximate the DW solution by a strongly convergent sequence. This task can be accomplished by taking the Ces\`aro averages over different mesh resolutions and proving what is
 commonly referred to as $\mathcal{K}$-convergence; see~Feireisl, Luk\'a\v{c}ov\'a-Medvid'ov\'a et al.~\cite{FLSY21}. As demonstrated in \cite[Theorem 10.5]{feireisl2021}, \textbf{strong convergence of the Ces\`aro averages} to a DW solution, as well as strong convergence of the approximate deviation of the associated Young measures, can be shown in this way. In this context, strong convergence of Ces\`aro averages means that, up to a subsequence of $u_{h_n}=(\varrho_{h_n},\bm{m}_{h_n},\eta_{h_n})$, we have
\begin{equation*}
\frac{1}{N} \sum_{n=1}^N u_{h_n} \to u \text{ as } N\to \infty \text{ in } L^q(\Omega\times (0,T), \mathbb{R}^{d+2}) \text{ for any } 1\leq q<\infty.
\end{equation*}
Furthermore, if we impose additional requirements on the regularity of the limit $u$, we can establish strong convergence by adapting \cite[Theorem 10.6]{feireisl2021} as follows.
\begin{theorem}[Strong convergence of the MCL scheme for the Euler system]\label{th_convergence}
  Let the sequence $\{u_h\}_{h \searrow 0}$ of numerical solutions
  $u_h=(\varrho_h, \bm{m}_h, \eta_h)^\top$ be generated using \eqref{mcl}
  on the interval $[0,T]$. Assume that the initial data $u_{h,0}=(\varrho_{h,0}, \bm{m}_{h,0},\eta_{h,0})^\top$ satisfy $\varrho_0 \geq \underline{\varrho}>0, (1-\gamma)\eta_0 \geq \underline{\varrho s}$, that the sequence of meshes $\{\mathcal{T}_h\}_{h \searrow 0}$ is shape regular, and that Lemma \ref{lemma1} is applicable. Finally, suppose that condition \eqref{vacuum} holds.
Then the following assertions are true:
\begin{itemize}
\item \textbf{Compatibility with weak solutions}:
  If a weak-(*) limit
  $u=(\varrho, \bm{m}, \eta)^\top$ of the sequence $\{u_h\}_{h\searrow 0}$ is a weak entropy solution corresponding to the initial data $u_0$, then
$\nu_{\mathbf{x},t}=\delta_{u(\mathbf{x},t)} $ for a.a. $(\mathbf{x},t) \in \Omega\times (0,T)$, and
\begin{align*}
(\varrho_h,\bm{m}_h, \eta_h) &\to (\varrho,\bm{m}, \eta) \text{ in } L^q(\Omega\times (0,T);\R^4),\\
E(u_h)=\frac12 \frac{|\bm{m}_{h}|^2}{\varrho_{h}} + \varrho_{h} e(\varrho_{h}, \eta_{h}) &\to \frac12 \frac{|\bm{m}|^2}{\varrho} + \varrho e(\varrho, \eta)  \text{ in } L^q(\Omega\times (0,T))
\end{align*}
for any $1 \leq q <\infty.$
\item \textbf{Compatibility with strong solutions}:
  Suppose that, for a given initial data~$u_0$,
  the Euler system admits a strong solution $u$ in the class
$\varrho, \eta \in W^{1,\infty}(\Omega\times (0,T)), \bm{m} \in W^{1,\infty}(\Omega\times (0,T); \R^d)$,
$\varrho \geq \underline{\varrho}>0$ in $\Omega\times[0,T]$.
Then for any $1\leq  q<\infty$ and $h\to 0$
\begin{align*}
(\varrho_{h},\bm{m}_h, \eta_h) &\to (\varrho,\bm{m}, \eta) \text{ in } L^q(\Omega\times (0,T); \R^4),\\
E(u_h)&\to E(u)  \text{ in } L^q(\Omega\times (0,T)).
\end{align*}
\item \textbf{Compatibility with classical solutions}:
Let $u=(\varrho, \bm{m}, \eta)$ be a weak solution such that
 $\varrho\in C^1( \overline{\Omega} \times [0,T])$, $\varrho\geq \overline{\varrho}>0,\; \bm{m}\in C^1(\overline{\Omega}\times [0,T];\R^d),\; \eta\in C^1 ( \overline{\Omega} \times [0,T])$. Then $u$ is a classical solution to the Euler system and
\begin{align*}
(\varrho_{h},\bm{m}_h, \eta_h) &\to (\varrho,\bm{m}, \eta) \text{ in } L^q(\Omega\times (0,T),\R^4)
\end{align*}
as $h\to 0$, for any $q\geq1$.
\end{itemize}
\end{theorem}
\begin{proof}[Sketch of the proof]
To show the compatibility with weak solutions, we notice that the defect $\mathfrak{E}$ vanishes, leading to the strong convergence of $E(u_{h_n})$ to $E(u)$ in $L^q(0,T;L^1(\Omega))$. Using the sharp form of the Jensen inequality, as stated in \cite[Lemma 7.1]{feireisl2021}, we deduce that $\nu_{\mathbf{x},t}=\delta_{u(\mathbf{x},t)}$ for almost every $(\mathbf{x},t) \in \Omega\times (0,T)$. The additional assumption that numerical solutions are uniformly bounded enables us to prove the strong convergence of $u_h$ to a weak solution $u$.

If a strong solution to the Euler system exists, we can apply the DW-strong uniqueness principle; see~Remark~5. Consequently, we have $\nu_{\mathbf{x},t}=\delta_{u(\mathbf{x},t)}$, $\mathfrak{R}=0$, and $u$ is a strong solution. Since the limit is unique, the entire sequence $u_h$ converges strongly to the strong solution.

The validity of the last claim follows directly from the compatibility between a DW and a classical solution for functions that meet our regularity assumptions; see~Remark~5.
\end{proof}

\section{Conclusions}
\label{sec:conclustions}

This work sets up a framework for numerical analysis of structure-preserving
finite element methods in the context of nonlinear hyperbolic problems. In
particular, we analyzed a semi-discrete scheme that uses a monolithic convex
limiting strategy to ensure preservation of invariant
domains and entropy stability. 
Our analysis revealed that the latter property imposes an upper bound on the spatial variation of numerical solutions. An important consequence of this result is the consistency of the MCL method for general hyperbolic conservation laws. Assuming strong convergence, we improved the Lax-Wendroff-type theorem presented in \cite{entropyFD} by exploiting entropy stability instead of making
assumptions that are too restrictive or difficult to verify. The main
result of this preliminary analysis is summarized in Theorem~\ref{TH:LW}.
%

In the second part of this paper, we focused our attention on the semi-discrete MCL scheme for the multi-dimensional Euler equations. To prove convergence, we adopted the framework of dissipative weak solutions and extended our consistency analysis to weakly convergent sequences in Theorem~\ref{th_consistency}. Exploiting the stability and consistency of the method under investigation, we were able to obtain both weak and strong convergence results for flux-corrected finite element discretizations of the Euler system (see ~Theorems \ref{eq:theorem_weak} and~\ref{th_convergence}, respectively).

 It is worth mentioning that the concept of dissipative weak solutions is not limited to the Euler equations, as it has also been derived for the compressible Navier--Stokes equations (see ~Feireisl, Luk\'a\v{c}ov\'a-Medvid'ov\'a et al.~\cite{feireisl2021}), viscous magnetohydrodynamics flows (see Li and She~\cite{li2022}), and viscous multi-component flows (see Jin, Novotn\'y et al.~\cite{Jin2021}).
 However, to the best of our knowledge, no existing analytical or numerical results are focused on the equations of ideal magnetohydrodynamics or on inviscid multi-component/multi-phase flows. We intend to explore extensions of dissipative weak solutions to such models in the future. The chances of proving weak convergence in the multi-component scenario are particularly high, as volume fractions  can be naturally expressed in terms of measures. We also believe that the concept of dissipative weak solutions can shed additional light on the steady-state convergence behavior of structure-preserving numerical methods for nonlinear hyperbolic systems.

\section*{Acknowledgements}

The work of D.K.
was supported by the German Science Foundation (DFG) within the framework of the priority research program SPP 2410 under the grant KU 1530/30-1.
M.L.-M. was partially funded by the Mainz Institute of Multiscale Modelling and by the DFG within SFB/TRR 146, project C5 and within the SPP 2410, project LU 1470/10-1.
P.Ö. was supported by the DFG within SPP 2410, project OE 661/5-1 
(525866748) and under the personal grant 520756621 (OE 661/4-1).
M.L.-M. and P.\"O. gratefully acknowledge the support of the Gutenberg Research College,  JGU Mainz.

\end{document}